\documentclass[a4paper,12pt,oneside]{article}

\usepackage{mathtools}
\usepackage{tikz}
\usetikzlibrary{shapes,arrows,positioning} \usepackage[framemethod=tikz]{mdframed}
\usepackage{soul}
\usepackage{transparent}
\usetikzlibrary{arrows, decorations.markings}
\usetikzlibrary{automata,positioning}
\usepackage[T1]{fontenc}
\usepackage{tikz}
\usetikzlibrary{shadings,shadows,shapes.arrows}
\usetikzlibrary{shapes,snakes}
\usetikzlibrary{matrix}
\usepackage{fancybox}
\usepackage{etex}
\usepackage{tcolorbox}
\usepackage{color}
\usepackage{fancybox,framed}

\usepackage[framemethod=tikz]{mdframed}
\usepackage{lipsum}
\usepackage{amssymb}
\usepackage{pifont}
\definecolor{mycolor}{rgb}{0.122, 0.435, 0.698}
\usepackage{emptypage}
\usepackage{afterpage}

\usepackage{epigraph}
\usepackage{url}
\usepackage{fancyhdr}
\usepackage{enumitem}
\usepackage{wrapfig}
\usepackage[toc,page]{appendix}
\usepackage{tikz}
\usepackage{fp}
\usepackage{pgf}
\usepackage{xifthen}

\usepackage[utf8]{inputenc}

\usepackage{color}
\usepackage{import}
\usepackage{setspace}
\usepackage{latexsym,amsmath,amsfonts,amscd,amssymb,amsthm}
\usepackage{graphicx, caption}
\usepackage[all]{xy}
\usepackage{epstopdf}
\DeclareGraphicsExtensions{.pdf,.png,.jpg,.gif, .eps}

\usepackage{amsthm}
\usepackage{thmtools}

\theoremstyle{plain} 
\newtheorem{teo}{Theorem}
\newtheorem{lem}{Lemma}

\newtheorem{cor}{Corollary}
\theoremstyle{definition}

\theoremstyle{remark}

\newcommand{\mapeo}[5]{
	\begin{eqnarray*}
		#1:#2 & \longrightarrow & #3\\
		#4 & \longmapsto & #5
\end{eqnarray*}}

\renewcommand{\epsilon}{\varepsilon}

\usepackage{hyperref}
\hypersetup{
	colorlinks=true,
	linkcolor=blue,
	filecolor=magenta,      
	urlcolor=cyan,
	pdfpagemode=FullScreen,
}

\usepackage[light, math]{iwona}
\usepackage[T1]{fontenc}
\makeindex
\makeatletter
\newcommand{\subjclass}[2][2010]{%
	\let\@oldtitle\@title%
	\gdef\@title{\@oldtitle\footnotetext{#1 \emph{Mathematics Subject Classification.} #2}}%
}
\newcommand{\keywords}[1]{%
	\let\@@oldtitle\@title%
	\gdef\@title{\@@oldtitle\footnotetext{\emph{Key words and phrases.} #1.}}%
}
\makeatother
\title{\bf{Hausdorff reflection preserves shape}}
\usepackage{authblk}
\author{Diego Mondéjar}
\date{}                     
\begin{document}
\maketitle
\begin{abstract}
We show that the Hausdorff reflection preserves the shape type of spaces. Some examples as well as the applicability in inverse limits of finite spaces are presented.
\end{abstract}

\section{Introduction}
Theree are some examples of studies in non-Hausdorff spaces many years ago, as the Alexandroff spaces in \cite{Adiskrete}. But It has been largely considered that non-Hausdorff spaces are not very interesting since they can hold only a very poor topological structure. It is very significant the inclusion of the $T_2$ condition in the definition of compact space in Bourbaki. In the other direction, we have Kelley's topology book \cite{Kelley}, where we can find many definitions not assuming spaces are Hausdorff. In \cite{Willif} and \cite{ReiNon} there are some methods proposed for dealing with non-Hausdorff spaces. More recently, there has been some renewed interest in these spaces because of the development of Digital Topology \cite{KonAto} and finite topological spaces \cite{MayFin,BarAlg}. This paper rows somehow in that direction. We want to compare the Hausdorff reflection of a topological space with the space itself in terms of shape type. As a motivation, we can cite \cite{Mon}, where it is shown that the Tychonoff functor indeed induces the identity morphism in shape. So, a topological space and its Tychonoff reflection have the same shape. We will show the same for the Hausdorff reflection.

We begin in Section \ref{sec:reflection} briefly reviewing the Hausdorff reflection of topological spaces and the theory of shapes, giving some insight about its utility in the handling of some pathological spaces. Then, in Section \ref{sec:result}, we prove the main theorem: the Hausdorff reflection of a topological space has the same shape as the original space. Finally, we offer some examples in Section \ref{sec:examples}, where this fact can be exploited to obtain some conclusions. It is specially relevant in the case of inverse limits of finite $T_0$ spaces, since they strengthens known results of reconstruction of topological spaces.
\section{Hausdorff reflections and Shape Theory}\label{sec:reflection}
In this section we recall the main two elements necessary to state our result.

The idea of a reflection of a topological space is to construct another space, as similar as posible to the first one, with an extra separation property and a universal map. It is similar in spirit to compactification. Concretely, given a topological space $X$ and a separation property $T$, we will say that $\mu_X:X\rightarrow X_T$ is the $T$ \emph{reflection} of $X$ if
$\mu_X$ is a continuous map, $X_T$ has the property $T$ and every continuous map $f:X\rightarrow Z$ with $Z$ having property $T$, factors through a map $g:X_T\rightarrow Z$, i.e., the diagram
$$\xymatrix @C=20mm{X \ar[d]_{\mu_X}\ar[r]^{f}& Z \\
X_T\ar[ur]_{g}}$$ commutes. If the map $\mu_X$ is surjective we will say that the reflection is surjective, too. The existence of surjective reflections for the separation properties $T_i,\enspace i=0,1,2,3,3\frac{1}{2}$ is a well known fact \cite{MNtopics}. It is easy to see that two reflections of the same space are homeomorphic.

In many cases, reflections are obtained as quotient spaces. But this is not always the case. For example, the Tychonoff functor -or reflection- can not be obtained with a quotient by a relation. Nevertheless, the relation is sometimes far from obvious. As a matter of fact, in order to obtain the Hausdorff reflection we need to define the following relations, as shown in \cite{SBMa}, a short and beautiful paper about reflections.
\begin{itemize}[itemsep=0ex]
\item $x R_1 y$ iff for every pair of neighborhoods $U_x, U_y$ of $x, y$ resp., we have $U_x\bigcap U_y\neq\emptyset$.
\item $x R_2 y$ iff there exist $x=z_1,z_2,\ldots,z_n=y$ such that $z_1R_1z_2R_1\ldots R_1z_n$.
\item $x R_3 y$ iff for every $f:X\rightarrow Z$, with $Z$ Hausdorff, we have $f(x)=f(y)$ for every pair $x R_2 y$.
\end{itemize}
Then, the Hausdorff reflection of $X$ is the quotient space $X_H=X/R_3$. Being a quotient or not, it is difficult to understand the process of reflection in the sense that we do not know which properties or structures of the space are preserved.

On the other hand, Shape Theory is an extension of Homotopy Theory introduced by Borsuk \cite{Bconcerning} in order to deal with some pathologies of topological spaces making homotopy not work correctly. Namely, it is useful when we have bad local properties as in the Warsaw Circle. Although is not intended to be used with spaces with bad separation properties, we will see that it can help sometimes. Shape theory is defined for compact metric spaces in its beginings and developed for arbitrary topological spaces later. They make use of invese systems of ANRs in order to define morphisms between the spaces. For details, see \cite{MSshape}. We have an extension of the homotopy category, enlarging the set of morphisms. Thus not every shape morphism is represented by a continuous function, but we have that every continuous function induces a shape morphism. From \cite{Mshapes}, we have the following useful characterization for a function to induce an isomorphism in the shape category. Let us use the following notation: for a pair of topological spaces $Z, R$, $[Z,R]$ is the set of homotopy classes of continuous functions from $Z$ to $R$. For a map $h:Z\rightarrow R$, we represent by $[h]$ its homotopy class.
\begin{teo}\label{teo:funindiso}
Let $X$ and $Y$ be topological spaces and $f:X\rightarrow Y$ a continuous map. Then $f$ is a shape equivalence (that is, the shape morphism induced by $f$ is an isomorphism in the shape category) if and only if, for every metric ANR $P$, the function \mapeo{f}{\left[ Y,P\right] }{\left[ X,P\right] }{\left[ h\right] }{\left[ h\cdot f\right] } is a bijection.
\end{teo}
\section{Result}\label{sec:result}
 We begin by the following technical lemma, needed in the proof.
\begin{lem}
The Hausdorff reflection of the product $X\times I$, where $I=[0,1]$, is homeomorphic to $X_H\times I$.
\end{lem}
\begin{proof}
Consider the continuous map $$\begin{array}{rcl}f:X\times I&
\longrightarrow &X_H\times I\nonumber\\(x,t)&\longmapsto &
(\mu_X(x),t),\nonumber\end{array}$$ which is a quotient map. Moreover, the space $X_H\times I$ is Hausdorff, so there exists a continuous surjective map $h:(X\times I)_H\rightarrow X_H\times I$
such that the diagram
$$\xymatrix @C=20mm{X\times I\ar[d]_{\mu_{X\times I}}\ar[r]^{f}& X_H\times I \\
(X\times I)_H\ar[ur]_{h}}$$ commutes. We see that $h$ is actually a homeomorphism. First of all, $h$ is a quotient map, because $f$ and $\mu_{X\times
I}$ are (\cite{Egeneral}, pag 91). Also, it is an injective map. Indeed, let
$[a],[b]\in(X\times I)_H$ such that $h([a])=h([b])=([z],t)$. Considering that
$\mu_{X\times I}$ is surjective, there exist $(x,t_1),(x,t_2)$ such that
$\mu_{X\times I}(x,t_1)=[a]$ and $\mu_{X\times I}(y,t_2)=[b]$. Because of the commutativity of the previous diagram we have that
\begin{eqnarray}
([x],t_1)=f(x,t_1)=h(\mu_{X\times
I}(x,t_1))=h([a])=([z],t)\nonumber\\
([y],t_2)=f(y,t_2)=h(\mu_{X\times I}(y,t_2))=h([b])=([z],t)\nonumber
\end{eqnarray}
so $[x]=[y]=[z]$ and $t_1=t_2=t$. For this concrete $t$,
we consider the commutative diagram $$\xymatrix @C=20mm{X\ar[r]^{id\times t}\ar[d]_{\mu_X}& X\times I\ar[d]^{\mu_{X\times I}} \\
X_H\ar[r]_g & (X\times I)_H},$$ which exists for being $\mu_{X\times
I}\circ(id\times t):X\rightarrow(X\times I)_H$ a continuous map to a Hausdorff space. We consider the images of $x,y$ by the two different maps of the diagram.
As $[x]=[y]$, we obtain that
$[a]=[b]$, so $h$ is injective. A quotient and injective map is a homeomorphism$\qedhere$
\end{proof} 
\begin{teo}\label{teo:shapereflection}
A topological space $X$ has the same shape than its Hausdorff reflection $X_H$.
\end{teo}
\begin{proof}
We show that for every topological space $X$, the Hausdorff reflection $\mu_X:X\rightarrow X_H$ induces the identity morphism in shape. In order to show this, we are going to use the characterization of identity morphisms in shape, Theorem \ref{teo:funindiso}. So, $\mu_X:X\rightarrow X_H$ is the identity morphism in shape if and only if the map 
$$\begin{array}{rcl}[X_H,P]&
\longrightarrow &[X,P]\\
h&\longmapsto & h\cdot f,\\ \end{array}$$ with $P$ being any metric ANR, is bijective.

It is surjective. Given a map $g:X\rightarrow P$,
with $P$ ANR and then, Hausdorff, there exists a map
$h:X_H\rightarrow P$ such that $g=h\cdot\mu_X$, that is, what we wanted.
It is injective: Let $h_1,h_2:X_H\rightarrow P$, with $P$ ANR, two continuous maps such that $h_1\cdot\mu_X$ y $h_2\cdot\mu_X$
are homotopic, i.e., there exists a continuous map, $G:X\times
I\rightarrow P$ such that $G(x,0)=h_1\cdot\mu_X(x)$ and
$G(x,1)=h_2\cdot\mu_X(x)$. Being $P$ Hausdorff, there exists a continuous map $F:(X\times I)_H\rightarrow P$ such that
$G=F\cdot\mu_{X\times I}$. Applying the previous lemma, we get
$\mu_{X\times I}=\mu_X\times id$, so we have that the following diagram commutes:
$$\xymatrix @C=20mm{X\times I\ar[r]^G\ar[d]_{\mu_X\times id}& P\\
X_H\times I\ar[ur]_F & }.$$ Then, for every $x\in X$, we have
\begin{eqnarray}
F([x],0)=G(x,0)=h_1\cdot\mu_X(x)=h_1([x])\nonumber\\
F([x],1)=G(x,1)=h_2\cdot\mu_X(x)=h_2([x]).\nonumber
\end{eqnarray}
So, $h_1$ and $h_2$ are homotopic$\qedhere$
\end{proof} 

\section{Examples}\label{sec:examples}
We analyze some non-Hausdorff spaces and its Hausdorff reflection in order to show that there can be some non trivial properties on those spaces that can be detected using this result. Specially, we focus in its use for inverse limits of $T_0$ finite spaces.

\paragraph{The punctured circle} We define a topological space in a similar way to the sometimes called "line with two origins", a classical non-Hausdorff space. Consider, with the complex numbers notation, the following two circles lying on the plane \[C_1=\lbrace e^{i\theta}:\theta\in[0,2\pi)\rbrace,\enspace C_2=\lbrace 2e^{i\theta}:\theta\in[0,2\pi)\rbrace.\] Let us consider the relation $e^{i\theta} R 2e^{i\theta}$ for every $\theta\in(0,2\pi)$. That is, we identify every pair of concentric points of the circles but one. Now, we consider the punturred circle, which is the quotient space $\mathcal{C}=(C_1\cup C_2)/R$. It is clear that $\mathcal{C}$ is not $T_2$ since the points $1$ and $2$ are not separable by open neighborhoods. It is clear that the Hausdorff reflection of the puntured circle $\mathcal{C}_H$ is obtained by identifying these two points and it is just a circle. By Theorem \ref{teo:shapereflection}, the shape types of these two spaces are the same and hence we see that shape is able to detect the main structure of $\mathcal{C}$ ignoring the conflict points of non-Hausdorffness.

\paragraph{Finite spaces} Topological spaces with a finite set of points are a good example of non-Hausdorff spaces with some interesting structures. It is shown in \cite{MccSin} that there are $T_0$ finite spaces with the same homotopy and homology groups as any polyhedron. Actually, a finite topological space satisfying the $T_1$ separation property is discrete. Hence, the Hausdorff reflection of every finite topological space, being Hausdorff, is discrete. We recover the well known fact that shape type of fnite topological spaces is the same as a finite set of points with the discrete topology. Here, shape is not able to capture any relevant structure. It is worth mentioning here the category constructed \cite{Chshape} to classify, at the same time, compact Hausdorff spaces by their shape type and finite topological spaces by their weak homotopy class.

\paragraph{Invese systems of finite spaces} This is the most adequate framework for the relation shown in Theorem \ref{teo:shapereflection}. The inverse limit of finite $T_0$ spaces may not be Hausdorff but, in some cases, it exhibits a rich topological structure that can be captured by the shape type. We show two examples here. In a series of papers, \cite{KWfinite, KTWthe} the authors showed that every compact Hausdorff space is the Hausdorff reflection of the inverse limit of an inverse system of finite spaces. The finite spaces involved in this proof are defined with a boolean algebra on the open sets of all possible coverings. The idea is to reconstruct some topological properties of the original space with the inverse limit. Here, the Hausdorff reflection lacks a clear interpretation and poses an obstacle to understanding the recovery process, as pointed out by the authors. By applying Theorem \ref{teo:shapereflection}, we can obtain the following refinement, replacing the Hausdorff reflection with the shape recovery.
\begin{cor}
Every compact Hausdorff space has the same shape as the inverse limit of an inverse system of finite spaces.
\end{cor}
As a second example, we have that in \cite{MMreconstruction}, the conditions of the original space are relaxed and it is shown a closer relation, namely that every compact metric space $X$ has the same homotopy type as the inverse limit $\mathcal{X}$ of an inverse sequence of finite spaces. There, it is also shown that the Hausdorff reflection of the inverse limit $\mathcal{X}_H$ is homeomorphic to the original space $X$. Hence, in this case, the Hausdorff reflection is preserving not only shape byt the homotopy type of the original space. We do not know when this preservation holds in general.

\paragraph{Acknowledgments} The author is grateful to his thesis advisor, M.A. Morón, for his help regarding these results.

\paragraph{Funding} This work has been partially supported by the research project PGC2018-098321-B-I00(MICINN). The author has been also supported by the FPI Grant BES-2010-033740 of the project MTM2009-07030 (MICINN).

\addcontentsline{toc}{chapter}{References}
\bibliographystyle{siam}
\bibliography{HausdorffReflection}

\vspace{1cm}
\noindent\textsc{Diego Mondéjar}\\
Departamento de Matemáticas\\
CUNEF Universidad\\
\texttt{diego.mondejar@cunef.edu}

\end{document}